\documentclass[11pt,leqno]{amsart}
\usepackage{amscd}
\usepackage{pb-diagram}
\usepackage{amssymb}
\usepackage{amsfonts}
\usepackage{latexsym}
\usepackage{verbatim}
\usepackage{amsthm}
\usepackage{amscd}
\usepackage{color}
\numberwithin{equation}{section}
\newcommand{\N}{\nabla}

\newcommand{\op}{\overline{\partial}}
\newcommand{\p}{\partial}

\theoremstyle{plain}
\newtheorem{theorem}{Theorem}[section]
\newtheorem{definition}[theorem]{Definition}
\newtheorem{lemma}[theorem]{Lemma}
\newtheorem{prop}[theorem]{Proposition}

\newtheorem{rem}[theorem]{Remark}

\newcommand{\ov}[1]{\overline{#1}}

\newcommand{\w}[1]{\wedge}

\newcommand{\e}{\begin{equation}}
\newcommand{\ee}{\end{equation}}
\renewcommand{\l}[1]{\label{#1}}

\newcommand\C{{\mathbb C}}
\newcommand\R{{\mathbb R}}
\newcommand\Z{{\mathbb Z}}

\newcommand\im{\,{\rm Im}\,}

\renewcommand\d{{\partial}}

\newcommand\f{{\varphi}}

\begin{document}
\title{A parabolic flow of balanced metrics}
\author{Lucio Bedulli and Luigi Vezzoni}
\date{\today}

\thanks{This work was supported by the project FIRB ``Geometria differenziale e teoria geometrica delle funzioni'',
 the project PRIN
\lq\lq  Variet\`a reali e complesse: geometria, topologia e analisi armonica" and by G.N.S.A.G.A. of I.N.d.A.M}
\address{Dipartimento di matematica \\ Universit\`a dell'Aquila\\
via Vetoio\\ 67100 L'Aquila\\ Italy}
\email{lucio.bedulli@dm.univaq.it}
\address{Dipartimento di Matematica G. Peano \\ Universit\`a di Torino\\
Via Carlo Alberto 10\\
10123 Torino\\ Italy}
 \email{luigi.vezzoni@unito.it}

\begin{abstract}
We prove a general criterion to establish existence and uniqueness of a short-time solution to an evolution equation involving ``closed" sections of a vector bundle, generalizing a method used by Bryant and Xu \cite{BryantXu} for studying the Laplacian flow in $G_2$-geometry. We apply this theorem in balanced geometry introducing a natural extension of the Calabi flow to the balanced case. We show that this flow has always a unique short-time solution belonging to the same Bott-Chern cohomology class of the initial balanced structure and that it preserves the K\"ahler condition. Finally we study explicitly the flow on the Iwasawa manifold.
\end{abstract}
\maketitle

\tableofcontents

%%%%%%%%%%%%%%%%% INTRODUZIONE COMMENTATA %%%%%%%%%%%%%%%%%%%%%%%%%%%%%%%%%%%%%%%%%

\section{Introduction}
%$$
%\Delta_{BC}:=\partial\ov{\partial}\ov{\partial}^*\partial^*+\ov{\partial}^*\partial^*\partial\ov{\partial}+\ov{\partial}^*\partial\partial^*\ov{\partial}+\partial^*\ov{\partial}\ov{\partial}^*\partial+\ov{\partial}^*\ov{\partial}+\partial^*\partial\,.
%$$

%$$
%\Delta_{A}:=\partial\partial^*+\ov{\partial}\ov{\partial}^*+\ov{\partial}^*\partial^*\partial\ov{\partial}+\partial\ov{\partial}\ov{\partial}^*\partial^*+\partial\ov{\partial}^*\ov{\partial}\partial^*+\ov{\partial}\partial^*\partial\ov{\partial}^*
%$$

%which are related by

%$$
%\Delta_{BC}(\partial\ov{\partial}\varphi)=d\left(\Delta_A\varphi\right)\,.
%$$

%Note that for a $d$-closed form $\psi$ we have
%$$
%\Delta_{BC}\psi=\partial\ov{\partial}\ov{\partial}^*\partial^*\psi\,.
%$$

The notion of balanced metric on a complex manifold has been introduced by Michelsohn in \cite{M}. Given a complex manifold $M^n$, a balanced metric $g$ on $M$ is an Hermitian metric whose fundamental 2-form $\omega$ satisfies
\begin{equation}\label{balanced}
d\omega^{n-1}=0\,;
\end{equation}
the pair $(M,g)$ is called a \emph{balanced manifold}. A balanced structure can be alternatively regarded as a closed {\em positive} real $(n-1,n-1)$-form $\varphi$ (see section \ref{balanced} below).
Moreover, balanced structures are characterized by the property
$$
\Delta_{\partial}f=\Delta_{\op}f=2\Delta_{d}f
$$
for every $f\in C^{\infty}(M,\C)$ (see \cite{Gaud}). An important feature of the balanced category is that, in the compact case, it is stable under modifications (see \cite{AB2}). This means that  a
modification $F\colon \tilde M\to M$ of a compact balanced manifold always admits a balanced metric; in particular modifications of compact K\"ahler manifolds admit balanced metrics when they do not admit any K\"ahler metric. From the viewpoint of special connections, an Hermitian metric is balanced if and only all the canonical Hermitian connections (according to the definition introduced by Gauduchon in \cite{Gaudbumi}) have the same Ricci forms (see \cite{GGP}).

\medskip Balanced metrics are mainly interesting on compact manifolds non-admitting K\"ahler structures. Nilmanifolds (or more generally solvmanifolds) provide examples of such spaces (see e.g. \cite{elsa, ugarte}). In particular, the Iwasawa manifold has balanced metrics and can be considered as a fundamental example.  An important class of examples is given by $6$-dimensional
twistor spaces, since the canonical Hermitian metric of the twistor space of an anti-self-dual $4$-dimensional Riemannian manifold is always balanced (see \cite{M}) and by torus bundles over Riemannian surfaces. More examples of balanced manifolds have been obtained in \cite{fuliyau} by considering  conifold transitions of Calabi-Yau threefolds;
these last examples include balanced structures on the connected sum of two or more copies of $S^3 \times S^3$.

\medskip
In this paper we focus on parabolic flows involving balanced metrics. Our idea consists in a generalization of the Calabi flow to the balanced case. The Calabi flow is an important flow in K\"ahler geometry which arises as the gradient flow of the Calabi functional
$$
\omega \mapsto \int_{M}\left(s_{\omega}\right)^2\,\frac{\omega^n}{n!}
$$
when it is  restricted to the the K\"ahler cone of an initial K\"ahler metric $\omega_0$
$$
C_{\omega_0}=\left\{\omega_0+i\d\op\,\phi>0\,\,:\,\,\phi\in C^{\infty}(M,\R)\right\}
$$
where $s_{\omega}$ is the scalar curvature of the metric $\omega$ (see \cite{chen}).
The flow has the following equation
\begin{equation}\label{calabi-flow}
\begin{cases}
\frac{\partial}{\partial t}\omega(t)=i\d\op s_{\omega(t)}\\
\omega(0)=\omega_0\,.
\end{cases}
\end{equation}
The evolution equation \eqref{calabi-flow} can be alternatively rewritten in terms of positive $(n-1,n-1)$-forms as
\begin{equation}\label{calabi-flow-balanced}
\begin{cases}
& \frac{\partial}{\partial t}\varphi(t)=i\partial\bar{\partial} *_t(\rho_t \wedge *_t\varphi(t))\\
& \varphi(0)=\varphi_0\,,
\end{cases}
\end{equation}
where $\rho_t$ is the Ricci form of $\varphi(t)$ and $\varphi_0=*_{\omega_0}\omega_0$. Flow \eqref{calabi-flow-balanced} still makes sense in the balanced case when we replace $\rho_t$  with the Ricci form of the Chern connection. But it turns out that \eqref{calabi-flow-balanced} is not still parabolic when it is extended to the context of balanced metrics and  because of this reason we modify of \eqref{calabi-flow-balanced} in the balanced case as follows. 

Let $M$ be an $n$-dimensional complex manifold and $\varphi_0$ be a positive closed $(n-1,n-1)$-form on $M$.
We consider the following flow of balanced structures $\varphi$ on $M$
\begin{equation}\label{theflow}
\begin{cases}
& \frac{\partial}{\partial t}\varphi(t)=i\partial\bar{\partial} *_t(\rho_t \wedge *_t\varphi(t))+(n-1)\Delta_{BC}\varphi(t)\\
&d\varphi(t)=0\\
& \varphi(0)=\varphi_0\,,
\end{cases}
\end{equation}
where $$\Delta_{BC} = \partial\ov{\partial}\ov{\partial}^*\partial^*+\ov{\partial}^*\partial^*\partial\ov{\partial}+\ov{\partial}^*\partial\partial^*\ov{\partial}+\partial^*\ov{\partial}\ov{\partial}^*\partial+\ov{\partial}^*\ov{\partial}+\partial^*\partial$$ is the modified Bott-Chern Laplacian and $*_t$ and $\rho_t$ are the Hodge operator and the Chern Ricci form induced by $\varphi(t)$.
The modified Bott-Chern Laplacian is an elliptic $4$-order operator whose kernel describes the Bott-Chern cohomology of an Hermitian manifold. The Bott-Chern cohomology is defined as
$$
H_{BC}(M)=\frac{\ker d }{\im(\partial\op)}\,,
$$
and plays an important role in non-K\"ahler complex geometry. 

Flow \eqref{theflow} can be seen as a flow of $G$-structures with torsion. Flows of this kind have been largely studied in the last few years (see for instance \cite{Bryant, Gill, Gri, Kar, liwang, streets-tian1, streets-tian2, streets-tian3, tosatti2, tosatti3, luigiflow, witt1,witt2,XuYe}). 

We will prove the following

\begin{theorem}\label{main}
The flow \eqref{theflow} admits a unique solution in the Bott-Chern class of $[\varphi_0]$ defined in a maximal interval $[0,\epsilon)$. Moreover if  the initial structure is K\"ahler then \eqref{theflow} reduces to the Calabi flow.
\end{theorem}

It turns out that the flow \eqref{theflow} is parabolic in a suitable sense when restricted to the space of balanced structures of a complex manifold. From this point of view \eqref{theflow}  is analogous to  the Laplacian flow introduced by Bryant in \cite{Bryant} in $G_2$ geometry.  Following the approach of Bryant-Xu in \cite{BryantXu} we prove a general criterion to establish short-time existence and uniqueness of the solution of an evolution equation involving \lq\lq closed '' sections of a vector bundle (see theorem \ref{general1} in section \ref{sectiongeneral1}).
The proof of theorem \ref{main} is then obtained as an application of our theorem \ref{general1}.

In the last section of the paper we study the flow on the Iwasawa manifold 
computing an explicit solution.

\medskip
%\noindent {\em{Acknowledgements}}. The authors would like to thank  Carlo Mantegazza for useful conversations about elliptic equations on smooth manifolds.
%During the  \lq \lq Workshop on geometric structures on manifolds and their applications'' (Castle Rauischholzhausen 2012)  the second author had an important conversation with Frederik Witt who observed a possible link between the flow considered in the paper and the Calabi flow. The authors are also very grateful to Enrico Priola for his fundamental help on understing  Hamilton's paper \cite{positive} and suggesting a possible generalization.  Finally, part of the paper has been writing during a visit of the first author to the university of Turin. The first author is grateful to the Turin University and the Politcenico of Turin for the hospitality.
\noindent {\em{Acknowledgements}}. The authors would like to thank  Carlo Mantegazza for useful conversations about parabolic equations on smooth manifolds and are very grateful to Enrico Priola for his fundamental help in understanding Hamilton's paper and to Valentino Tosatti for useful conversations.
Moreover the second author wishes to thank Frederik Witt who, during a useful an important conversation, observed a possible link between the flow considered in the paper and the Calabi flow.  Part of the work has been done during a visit of the first author to the University of Turin. The first author is grateful to the University and to the Politecnico of Turin for their hospitality. Finally we would like to thank an anonymous referee for helping us improve the exposition of the paper.

\section{Preliminary features of balanced metrics}
\subsection{Algebraic features}
Consider on $\R^{2n}$ the standard (linear) complex structure $J_0$ and let $\Lambda^{1,1}_+$ be the set of
%non-degenerate real forms of type $(1,1)$ compatible to the positive orientation.
$J_0$-positive $(1,1)$-forms. Every $\omega\in \Lambda^{1,1}_+$ determines the Hermitian metric
$$
g(X,Y)=\omega(X,J_0Y)
$$
and a  Hodge \lq\lq star'' operator $*\colon \Lambda^{r,s}\to \Lambda^{n-s,n-r}$ defined by
\[
\alpha \wedge *\overline{\beta} = g(\alpha,\bar{\beta}) \frac{\omega^n}{n!}\,.
\]
Note that
$$
*\omega=\frac{1}{(n-1)!}\omega^{n-1}\,.
$$
Let $\mathcal Q \colon \Lambda^{1,1}_+\to  \Lambda^{n-1,n-1}$ be the smooth map
$$
\omega\mapsto *_{\omega}\omega\,,
$$
where we emphasise the dependence of $*$ on the choice of the metric, and let
$$
\Lambda_+^{n-1,n-1}=\mathcal Q (\Lambda_+^{1,1})
$$
be the set of {\em positive} $(n\!-\!\!1,\!n\!-\!\!1)$-forms. It is well known that $\mathcal Q$ is a diffeomorphism onto its image (see \cite{M}).

Now let us fix $\omega\in \Lambda^{1,1}_+$. Then $\Lambda^{1,1}$ splits as
\begin{equation}
\label{dec(1,1)}
\Lambda^{1,1}=\mathbb{C}\,\omega\oplus\Lambda^{1,1}_0
\end{equation}
where
$$
\Lambda^{1,1}_0=\{\sigma\in \Lambda^{1,1}\,\,: \,\,\sigma\wedge \omega^{n-1}=0\}\,
$$
is the set of {\em primitive} (1,1)-forms.
In the same way $\Lambda^{n-1,n-1}$ splits as

$$
\Lambda^{n-1,n-1}=\mathbb{C}\,\varphi+\Lambda^{n-1,n-1}_0
$$
where $\varphi=*\omega$ and
$$
\Lambda^{n-1,n-1}_0=\left\{\gamma\in \Lambda^{n-1,n-1}\,\,: \,\,\gamma\wedge \omega=0\right\}\,.
$$
The following is an algebraic lemma which will be useful in the sequel. The proof is a straightforward computation.
\begin{lemma}
\label{star(1,1)}
Let $\sigma$ be in $\Lambda^{1,1}_0$, then for $0\leq k \leq n-2$
$$
*(\sigma \wedge \omega^k)=-\frac{1}{(n-2-k)!}\,\sigma\wedge \omega^{n-2-k}\,.
$$
In particular if $h=h_1\omega+h_0\in \Lambda^{1,1}$, then
$$
*h=\frac{1}{(n-1)!}h_1\,\omega^{n-1}-\frac{1}{(n-2)!}h_0\wedge \omega^{n-2}\,.
$$
\end{lemma}

%\begin{lemma}\label{iota}
%For $X\in \R^{2n}$  we have
%$$
%*(\iota_X g\wedge \omega)=\iota_{X}*\omega
%$$

%\end{lemma}
%\begin{lemma}
%Let $\theta$ be in $\Lambda^{1,0}$. Then
%$$
%*(\theta\wedge\omega^{n-1})=\frac{1}{(n-1)!}\theta\,.
%$$
%\end{lemma}

\subsection{Balanced structures on complex manifolds}\label{balanced}
Let $M^n$ be a complex manifold.
We denote by $\Lambda^{r,s}$ the bundle of complex forms of type $(r,s)$ on $M$ and by $C^\infty(M,\Lambda^{r,s})$ the vector space of its smooth sections.

Following the notation of the previous subsection, the following fiber bundles on $M$ are naturally defined
$$
\Lambda_+^{1,1}\,,\quad \Lambda_+^{n-1,n-1}\,.
$$

A {\em balanced} structure on $M$ is a global coclosed section $\omega$ of $\Lambda_+^{1,1}$, or, equivalently, is a closed section $\varphi$ of $\Lambda_+^{n-1,n-1}$. For the geometry of balanced manifolds we refer to \cite{M,AB1,AB2,fuyau,fuliyau,tosatti,saracco} and the references therein.
\begin{rem}{\em
We remark that if the fundamental form $\omega$ of an Hermitian metric $g$ on an $n$-dimensional complex manifold satisfies $d\omega^{k}=0$,  for some $k<n-1$, then $\omega$ is closed and $g$ is a K\"ahler metric. Therefore from the point of view of the powers of $\omega$, the balanced condition is the unique possible generalization of the K\"ahler one. On the other hand, another important generalization of K\"ahler structures is given by SKT metrics which are defined as Hermitian structures whose fundamental form is $\d\op$-closed. As remarked in \cite{ivanov,finosalamonparton} balanced and SKT structures are two generalizations transverse to each other since  a balanced structure is also SKT if and only if it  is K\"ahler.}
\end{rem}

\subsection{The Chern connection}\label{the chern connection}
Given an Hermitian manifold  $(M,g,J)$, the {\em Chern connection} is defined as the unique Hermitian connection $\N$
%satisfying
%$$
%\N g=0\,\quad \N J=0\,,\quad {\rm Tor}^{1,1}=0\,.
%$$
%The first two conditions says that $\N$ has holonomy in ${\rm U}(n)$, %whilst the third one means that the $(1,1)$-component of the torsion of
%$\N$ vanishes.
whose torsion has vanishing $(1,1)$-component.
This connection is {\em canonical} according to the terminology introduced by Gauduchon in  \cite{Gaudbumi} and has an important role in complex geometry. In the K\"ahler case it coincides with the Levi-Civita connection. The curvature tensor $R$ of $\N$ is defined as usual
$$
R(X,Y,Z,W)=g([\N_X,\N_Y]Z,W)-g(\N_{[X,Y]}Z,W)\,.
$$
Moreover it is defined the {\em Chern Ricci form}
$$
\rho(X,Y):=\frac12 \sum_{k=1}^{2n}R(X,Y,Je_k,e_k)
$$
where $\{e_k\}$ is an arbitrary orthonormal frame.
In complex notation we  can write
$$
\rho=i\rho_{k\bar l} dz^k\wedge d\bar{z}^l
$$
where
$$
\rho_{k \bar{l}}=-\frac{\partial^2}{\partial z^k \partial \bar{z}^l} \,{\rm log}\,G\,,
$$
$G$ being the determinant of the matrix $g_{i\bar{j}}=g(\frac{\partial}{\partial{z^i}},\frac{\partial}{\partial{\bar{z}^j}})$. Finally we denote by $s=g^{k\bar{l}}\rho_{k\bar{l}}$ the scalar curvature of $\N$.

\subsection{Bott-Chern and Aeppli cohomology complexes}\label{BCsection} On a complex manifold $(M,J)$ beside the standard de Rham and Dolbeault theory, two more cohomology complexes are worth considering.
The {\em Bott-Chern cohomology groups} are defined as
$$
H^{p,q}_{BC}(M)=\frac{\ker(d: C^\infty(M,\Lambda^{p,q}) \to C^\infty(M,\Lambda^{p+q+1}\otimes\C))}{\im(\partial\op: C^\infty(M,\Lambda^{p-1,q-1}) \to C^\infty(M,\Lambda^{p,q}))}\,,
$$
while the {\em Aeppli cohomology groups} are\\
%\textcolor{magenta}{bisogna trovare un trucco per accorciare questa riga, forse togliere i $C^\infty$ in tutta la sezione}
$$
H^{p,q}_{A}(M)=\frac{\ker(\partial\op: C^\infty(M,\Lambda^{p,q}) \to C^\infty(M,\Lambda^{p+1,q+1}))}
{\partial\left(C^\infty(M,\Lambda^{p-1,q})\right) +\op\left( C^\infty(M,\Lambda^{p,q-1})\right) }
$$
see \cite{Angella,Demailly,bigolin,Schweitzer}.
As in the case of de Rham cohomology groups, when $M$ is compact, Bott-Chern and  Aeppli cohomology groups are isomorphic to the kernel of suitable linear differential operators acting on forms. More precisely, as soon as an Hermitian metric is fixed on $M$ we can define the {\em modified Bott-Chern Laplacian} (see \cite{Schweitzer})
$$\Delta_{BC} = \partial\ov{\partial}\ov{\partial}^*\partial^*+\ov{\partial}^*\partial^*\partial\ov{\partial}+\ov{\partial}^*\partial\partial^*\ov{\partial}+\partial^*\ov{\partial}\ov{\partial}^*\partial+\ov{\partial}^*\ov{\partial}+\partial^*\partial\,,$$
and the {\em modified Aeppli Laplacian}
$$\Delta_{A}:=\ov{\partial}^*\partial^*\partial\ov{\partial}+\partial\ov{\partial}\ov{\partial}^*\partial^*+\partial\ov{\partial}^*\ov{\partial}\partial^*+\ov{\partial}\partial^*\partial\ov{\partial}^*+\partial\partial^*+\ov{\partial}\ov{\partial}^*\,.$$
They both are fourth order elliptic operators on $C^\infty(M,\Lambda^{p,q})$ and define a Hodge-like decomposition. We will need just the decomposition induced by $\Delta_A$.
\begin{theorem}[\cite{Schweitzer}] If $(M,g,J)$ is a compact Hermitian manifold then we have the following
orthogonal decomposition for every $(p,q)$
$$
C^\infty(M,\Lambda^{p,q})=\ker \Delta_A \oplus (\im \partial + \im \op) \oplus \im(\partial\op)^*\,.
$$
\end{theorem}

%\textcolor{magenta}{Decidere se metter qui questa proposizione, oppure nella dimostrazione del teorema}\\
The following proposition is an important step in the proof of theorem \ref{main}.
\begin{prop}\label{Fisdirectsummand} Let $G_A$ be the Green operator associated to the modified Aeppli Laplacian. Then for every $\psi\in \partial\op\,C^{\infty}(M,\Lambda^{p,q})$ we have
%Consider the following diagram
%$$
%\begin{CD}
%C^{\infty}(M,\Lambda^{p,q}) @>\partial\op>>C^{\infty}(M,\Lambda^{p+1,q+1})\\% @>d>>C^{\infty}(M,\Lambda^{})\\
%@VV\Delta_AV \\%@VV\Delta_{BC}V\\
%C^{\infty}(M,\Lambda^{p,q}) @<(\partial\op)*<<C^{\infty}(M,\Lambda^{p+1,q+1})\,,
%\end{CD}
%$$
$$
\psi=\partial\op \,G_A(\partial\op)^*(\psi)\,.
$$
\end{prop}

\begin{proof} Let $\psi\in \partial\bar{\partial}\Lambda^{p,q}$. Bearing in mind the Aeppli decomposition we can write
$$
\psi=\partial \op \beta\quad {\rm with} \quad
\beta\in {\rm Im}(\partial\op)^*\,.
$$
In particular we have
$$
\Delta_A\beta= (\partial \op)^*\partial\op \beta
$$
and consequently
$$
\beta=G_A((\partial \op)^*\partial\op \beta)=G_A((\partial \op)^*\psi)\,.
$$
Thus finally
$$
\psi=\partial \op G_A(\partial \op)^*\psi\,,
$$
as required.
\end{proof}

\subsection{Families of balanced structures}
Here we consider a smooth $1$-parameter family $\varphi$ in $C^\infty(M,\Lambda^{n-1,n-1}_+)$ and compute the derivative of $*_\varphi \varphi$ which will be useful when we will study flows of balanced structures.
\begin{lemma}\label{firstvariation}
Assume
$$
\frac{d}{dt}\varphi=h_1\,\varphi+*_\varphi h_0\,,
$$
whith $h_1\in C^{\infty}(M,\R)$ and $h_0\in C^{\infty}(M,\Lambda^{1,1}_0)$.
Then
$$
\frac{d}{dt}(*_\varphi \varphi)=\frac{h_1}{n-1}*_\varphi \varphi-h_0\,.
$$
\end{lemma}

\begin{proof}
We can write
$$
\frac{d}{dt}(* \varphi)=f_1*\varphi+f_0
$$
with $f_1\in C^\infty(M,\R)$ and $f _0\in C^{\infty}(M,\Lambda^{1,1}_0)$. (Here and in the rest of the proof the symbol $*$ will always stand for $*_\varphi$).
Recalling that
$$
\varphi=\frac{1}{(n-1)!}(*\varphi)^{n-1}
$$
we can compute
$$
\frac{d}{dt}\varphi=\frac{1}{(n-2)!} (*\varphi)^{n-2}\wedge \frac{d}{dt}(*\varphi)=\frac{1}{(n-2)!} (f_1*\varphi+f_0) \wedge (*\varphi)^{n-2}\,.
$$
Using lemma \ref{star(1,1)} we get
$$
\frac{d}{dt}\varphi=(n-1)f_1\varphi-*f_0\,,
$$
hence
$$
f_1=\frac{1}{n-1}h_1\,,\quad f_0=-h_0\,,
$$
as required.
\end{proof}

%\section{Description of the flow}
%In this short section we describe the problem \eqref{theflow}.
%
%\smallskip
%Let $(M,g_0,J)$ be a compact $n$-dimensional balanced manifold. We denote by $\omega_0$ the fundamental form of $g_0$ and $\varphi_0=*_{\omega_0}\omega_0$.  Then, as explained in the introduction, we consider the evolution equation \eqref{theflow}
%\begin{equation*}
%\begin{cases}
%& \frac{\partial}{\partial t}\varphi(t)=i\partial\bar{\partial} *_t(\rho_t \wedge *_t\varphi(t)) + (n-1)\Delta_{BC}\varphi(t)\\
%&d\varphi(t)=0\\
%& \varphi(0)=\varphi_0
%\end{cases}
%\end{equation*}
%and we seek for a solution $\varphi(t)$  lying for every $t$ in the Bott-Chern cohomology class of $\varphi_0.$ Therefore we can write
%$$
%\varphi(t)=\varphi_0+\beta(t)
%$$
%and try to solve the problem
%\begin{equation}\label{flowbeta}
%\begin{cases}
%& \frac{\partial}{\partial t}\beta(t)=i\partial\bar{\partial} *_t(\rho_t \wedge *_t(\varphi_0+\beta(t))) + (n-1)\Delta_{BC}(\varphi_0+\beta(t))\\
%&d\beta(t)=0\\
%& \beta(0)=0\,,
%\end{cases}
%\end{equation}
%where for every $t$ the form $\beta(t)$ is $\partial\op$-exact and such that $\varphi_0 + \beta(t)$ is still positive.

\section{Flows of closed sections of vector bundles}\label{sectiongeneral1}
The Laplacian flow has been introduced in \cite{Bryant} as a natural tool for studying closed $G_2$-structures. In \cite{BryantXu} Bryant and Xu proved that the solution to the Laplacian flow of  $G_2$-structures exists and is unique in a short interval of time. The proof of this result was obtained in \cite{BryantXu} by the following two steps:  \\
the first step consists in applying a sort of De Turk's trick to the Laplacian flow in order to have a new flow  which is  parabolic along the direction of closed $G_2$-structures and it is equivalent to the initial one;  the second step consists in applying some analytic techniques introduced by Hamilton for studying parabolic equations
in order to prove the result for the new flow.

\medskip
The aim of this section is to give a generalization of the second part of the Bryant-Xu proof mentioned above to the general case of a parabolic evolution equation associated to an operator acting on fiber bundles.
\subsection{Flows on compact manifolds}
Let us consider the following
\begin{definition}
A {\rm Hodge system} on a manifold $M$ consists of the following sequence
\begin{equation}\label{hodge complex}
\begin{CD}
C^{\infty}(M,E_-) @>D>>C^{\infty}(M,E) \\
@VV\Delta_DV\\% @VV{\rm Id}V\\
C^{\infty}(M,E_-) @<D^*<<C^{\infty}(M,E)
\end{CD}
\end{equation}
where $E_-$ and $E$  are fiber bundles over $M$ with an assigned metric along their fibers, $D$ is a differential operator, $D^*$ is the formal adjoint of $D$  and $\Delta_D$ is an elliptic operator such that
\begin{equation}\label{G}
\psi=DGD^*\psi
\end{equation}
for every $\psi\in {\rm Im}\,D$, where $G$ is the Green operator of $\Delta_D$.
\end{definition}

%%%%%%%%%%%% REMARK COMMENTATO %%%%%%%%%%%%%%
\begin{comment}
Notice that require \eqref{G} is less then having the following  commutative diagram
\begin{equation}\label{hodge complex}
\begin{CD}
C^{\infty}(M,E_-) @>D>>C^{\infty}(M,E) \\%@>D_+>>C^{\infty}(M,E_+)\\
@VV\Delta_DV @VV{\rm Id}V\\
C^{\infty}(M,E_-) @<D^*<<C^{\infty}(M,E)
\end{CD}
\end{equation}
\end{comment}
%%%%%%%%%%%%%%%%%%%%%%%%%%%%%%%%%%%%%%%
\begin{comment}
\textcolor{red}{If $(M,g)$ is a compact Riemaniann manifold, the following is a Hodge complex
\begin{equation}
\begin{CD}
C^{\infty}(M,\Lambda^{p-1}) @>d>>C^{\infty}(M,\Lambda^{p})\\
@VV\Delta=dd^*+d^*d V\\
C^{\infty}(M,E_-) @<d^*<<C^{\infty}(M,E)
\end{CD}
\end{equation}
Moreover, proposition \ref{Fisdirectsummand} implies that the Aeppli Laplacian operator $\Delta_{A}$ induces a Hodge complex.}
\end{comment}
%%%%%%%%%%%%%%%%%%%%%%%%%%%%%%%%%%%%%%%
We have the foremost example of Hodge system taking $E_-=\Lambda^p$, $E=\Lambda^{p+1}$, $D=d$ and $\Delta_D=dd^*+d^*\!d$, the standard Laplace operator, on a compact Riemannian manifold. Condition \eqref{G} in this case is a consequence of standard Hodge theory.

%\textcolor{red}{Il "nostro" Hodge system lo mettiamo gia' qui o nella dimostrazione del main? Io lo lascerei nel teorema, ma non so.}

\medskip
Consider a Hodge system on a compact manifold $M$ as in the previous definition. Let
$A$ be an open subset of $E$ such that $\pi(A)=M$, where $\pi\colon E\to M$ is the projection.
Consider a non-linear partial differential operator of order $2m$
$$
L\colon C^{\infty}(M,A)\to C^{\infty}(M,E)
$$
and a fixed initial datum $\varphi_0\in C^{\infty}(M,A)$ such that
\begin{equation}\label{compatibility}
L(\varphi_0+D\gamma)\in {\rm Im}\,D
\end{equation}
for every $\gamma\in C^{\infty}(M,E_-)$.
It will be convenient to think of $L$ as extended in the trivial way to time-dependent sections of $A$.
Then consider the evolution problem
\begin{equation}\label{flow}
\begin{cases}
& \frac{{\partial}}{{\partial}t}\varphi=L(\varphi)\\
& \varphi(0)=\varphi_0\,,
\end{cases}
\end{equation}
where the solution $\varphi(t)$ is sought  in the space
$$
U=\{\varphi_0+D\gamma \,\,:\,\,\gamma\in C^\infty(M,E_-)\}\,\cap C^{\infty}(M,A)
$$
and is required to depend smoothly on time.  Notice that
problem \eqref{flow} makes sense because of condition \eqref{compatibility}.
Let $\mathcal{D}^{2m}(E,E)$ denote the space of partial differential operators on $E$ of order $\leq 2m$. Recall that $\mathcal{D}^{2m}(E,E)$ can be seen as the space of smooth sections of a vector bundle (see e.g. \cite{HamiltonNash}) and that a linear partial differential operator $Q$ of order $2m$ is said to be {\em strongly elliptic} if its principal symbol $\sigma_Q(x,\xi)$ satisfies the following inequality
$$-\langle\sigma_Q(x, \xi)v, v\rangle_E \geq \lambda |\xi|^{2m}|v|^{2m} $$
for some positive constant $\lambda$ and for all $(x,\xi)\in TM$, $\xi \neq 0$ and $v \in E_x$ (The definition actually does not depend on the metric $\langle\cdot,\cdot\rangle_E$ along the fibers of $E$).
Here the principal symbol of $Q$ is defined by
$$
\sigma_Q(x,\xi)v = \frac{i^{2m}}{(2m)!} Q(f^{2m}u)(x)
$$
for an $f \in C^\infty(M)$ with $f(x) = 0$, $d_x f = \xi$ and $u \in C^\infty(E)$ with $u(x) = v$.

Finally we denote by $L_{*|\varphi}$ the derivative of the operator $L$ at $\varphi$.

Now we can state the following theorem where we denote by $L_{*|\varphi}$ the derivative of the operator $L$ at $\varphi$.

\begin{theorem}\label{general1}
Let $(E_-,E,D,\Delta_D)$ be a Hodge system on a compact Riemannian manifold $M$. Let $L$, $\varphi_0$ and $U$ be as above. Assume that there exists a nonlinear partial differential operator
$$
\tilde{L}\colon C^\infty(M,A)\to \mathcal{D}^{2m}(E,E)\,, \quad \varphi \mapsto \tilde{L}_\varphi
$$
such that
\begin{enumerate}
\vspace{0.1cm}
\item[1.] $\tilde L_\varphi$ is strongly elliptic for every $\varphi\in U;$

\vspace{0.1cm}
\item[2.] $L_{*|\varphi}(\psi)=\tilde L_{\varphi}(\psi)$ for every $\varphi\in U$ and $\psi\in DC^{\infty}(M,E_{-})$.
\end{enumerate}Assume further that
$$
L_{*|\varphi}(D\theta)=Dl_{\varphi}(\theta)\,.
$$
for every $\theta \in C^\infty(M,E_-)$, where $l_\varphi$ is a strongly elliptic linear differential operator on $E_-$. Then there exists $\epsilon > 0$ such that the system \eqref{flow} has a unique solution $\varphi\in C^\infty([0,\epsilon), U)$.
\end{theorem}
%\textcolor{red}{In the statement of theorem \ref{general1} above, $\mathcal{D}^{2m}(E,E)$ denotes the space of partial differential operator on $E$ of order $\leq 2m$ which can be seen itself as the space of smooth sections of a vector bundle (see e.g. \cite{HamiltonNash}).}
\subsection{The Laplacian flow for $G_2$-structures}
In order to clarify our notation, it is useful to describe the objects involved in proposition \ref{general1} in the
case of the Laplacian flow for $G_2$-structures studied in \cite{BryantXu}.

In the $G_2$-case one has the following Hodge system: a compact $7$-manifold $M$ admitting closed $G_2$-forms, $E_-=\Lambda^2$, $E=\Lambda^3$, the operator
$$
D=d\colon C^{\infty}(M,\Lambda^2)\to C^{\infty}(M,\Lambda^3)\,,
$$
$A$ is the open set of $\Lambda^3$ whose global sections are $G_2$-structures on $M$ and
$\varphi_0$ is a fixed closed section of $A$. Furthermore $U$ is the set of $G_2$-structures lying in the cohomology class of $\varphi_0$ and
$$
L\colon C^\infty(M,A)\to C^{\infty}(M,\Lambda^3)
$$
is the operator
$$
L(\varphi)=\Delta_{\varphi}\varphi+\mathcal{L}_{\xi(\varphi)}\varphi\,,
$$
where $\Delta_{\varphi}$ is the standard Hodge Laplacian operator of the metric induced by $\varphi$
and $\xi(\varphi)$ is a suitable vector field depending on the torsion of $\varphi$.
Finally
$$
\tilde L_{\varphi}=-\Delta_{\varphi}+ d\Phi_{\varphi}
$$
and
$$
l_{\varphi}=-\Delta_{\varphi}+\Phi_{\varphi}
$$
where $\varphi\in U$ and  $\Phi_\varphi$ is a suitable algebraic linear operator with coefficients depending on the torsion of $\varphi$ in a universal way.
\begin{rem}
{\em In the same way the modified coflow of $G_2$-structures considered in \cite{Gri} can be described in terms of the general setting of theorem \ref{general1}.}
\end{rem}
%%%%%%%%%%%%%%%%%%%%%%%%%%%%%%%%%%%%%%%
%%% NON TOGLIERE GLI SPAZI FRA LE SEZIONI !!! %%%%%%%%%%%

\section{Proof of theorem \ref{general1}}
We need some preliminaries about the function spaces that will be needed in the proof of theorem \ref{general1}.
\subsection{Tame Fr\'echet spaces and tame maps}
Here we recall some basic facts about tame Fr\'echet spaces and tame maps.
For a detailed description of these topics we refer to \cite{HamiltonNash} .

\smallskip
A {\em tame Fr\'echet space} is by definition a vector space $V$ endowed with an increasing countable family of seminorms $\{|\cdot|_n\}$. The family $\{|\cdot|_n\}$ gives a topology on $V$ by defining a sequence $\{x_n\}\subseteq V$ {\em convergent} if it converges in each seminorm. A continuous map $F\colon (V,|\cdot|_n)\to (W,|\cdot|'_n)$ between two tame Fr\'echet spaces is called {\em tame} if any $x\in V$ has a neighborhood $U_x$ such that
$$
|F(y)|'_n\leq C_n(1+|y|_{n+r})
$$
for every $y\in U_x$ and $n>b$. Here $b$,$r$ and $C_n$ are allowed to depend on the neighborhood. A differentiable map is called {\em smooth tame} if all its derivatives are tame maps.
The next result is known in the literature as the {\em Nash-Moser  theorem.}

\begin{theorem}[\bf{Nash-Moser}]
Let $\mathcal{F}$, $\mathcal{G}$ be tame Fr\'echet spaces and let  $\mathcal{U}$
be an open set of  $\mathcal{F}$. Consider a smooth tame map $F\colon \mathcal U\to \mathcal G$ such that the first variation of $F$, $F_{*|f}\colon \mathcal F\to \mathcal G$, is an isomorphism  for every $f\in\mathcal{U}$ and such that the map
$VF\colon \mathcal{U}\times \mathcal{G}\to \mathcal{F}$ defined as $VF(f,g)=F_{*|f}^{-1}g$
is smooth tame. Then $F$ is locally invertible with local inverses smooth tame.
\end{theorem}

\subsection{Tame Fr\'echet spaces and sections of fibre bundles.} Here we introduce the tame Fr\'echet spaces we want to consider.
Let $M$ be a compact manifold and let $\pi\colon E\to M$ be a vector bundle over $M$ with a metric along its fibres. Then the space $C^{\infty}(M,E)$ of global smooth sections of $E$  inherits a natural structure of tame Fr\'echet space defining $|e|_n$ as the $L^2$ norm of $e$ and its covariant derivatives up to degree $n$.
Fix now a real number $T>0$ and consider the space of time-dependent partial differential operators $P\colon C^{\infty}(M \times [0,T],E)\to C^{\infty}(M \times [0,T],E)$
of degree less than or equal to $r$.  This space is tame Fr\'echet with respect to the grading
$$
|[P]|_n=\sum_{jr\leq n} [(\partial/\partial t)^{j}P]_{n-jr}
$$
where $[P]_n$ is the sup of the norm of $P$ and its space derivatives up to degree $n$.

Furthermore the space $C^{\infty}([0,T]\times M, E)$ of curves of smooth sections of $E$ is a tame Fr\'echet space with respect to the grading
\begin{equation}\label{grading}
\|e\|_n=\sum_{jr\leq n}|(\partial/\partial t)^{j}e|_{n-jr}\,,
\end{equation}
where for every time-dependent section $f$ of $E$ we put
$$
|f|_n^2=\int_{0}^{T}|f_t|_n^2\,dt\,
$$
and $f_t \in C^\infty(M,E)$ is defined by $f_t(x)=f(t,x)$.
Now we are ready to prove theorem \ref{general1}.

\begin{proof}[Proof of theorem $\ref{general1}$]
%Fix $T>0$ and  let
%$$
%\mathcal{F}=dC^{\infty}([0,T]\times M,\Lambda^{k-1})\,,\quad \mathcal{G}=\mathcal{F}\times dC^{\infty}(M,\Lambda^k)\,.
%$$
%In view of Proposition 4.2 of \cite{BryantXu} $\mathcal{F}$ is a tame Fr\'echet space with respect to %the grading $\|\cdot\|_n$ \eqref{gradingt} and $\mathcal{G}$ is tame Fr\'echet space with respect to %$\|\cdot\|_n+|\cdot|_n$.
Let us put
$$
\mathcal{U}=\{\beta\in DC^{\infty}(M\times [0,T],\,E_{-})\,\,:\,\, \varphi_0+\beta(t)\in U \mbox{ for every }t\in [0,T]\}
$$
and consider the evolution problem
\begin{equation}
\label{betaflow}
\begin{cases}
\frac{\partial }{\partial t}\beta=L(\varphi_0+\beta)\\
\beta(0)=0\,.
\end{cases}
\end{equation}
If $\beta$ is a solution of \eqref{betaflow} then $\varphi=\varphi_0+\beta$ is a solution of \eqref{flow}.

Let us consider the map
$$
F\colon \mathcal{U}\to C^{\infty}(M\times [0,T],E)\times DC^{\infty}(M,E_{-})
$$
defined by
$$
F(\beta)=\left(\frac{\d}{\d t} \beta-L(\varphi_0+\beta),\beta(0) \right)\,.
$$
The proof mainly consists in the application of the Nash-Moser inverse theorem to $F$.
The derivative of $F$ at $\beta$ is the linear map
$F_{*|\beta}\colon \mathcal{F}\to \mathcal{G}$ given by
$$
F_{*|\beta}(\psi)=\left(\frac{\partial}{\partial t} \psi-L_{*|\varphi_0+\beta}(\psi),\psi(0)\right)\,,
$$
where
$$
\mathcal{F}=DC^{\infty}(M\times [0,T],E_-)
$$
and
$$
\mathcal{G}=\mathcal{F}\times DC^{\infty}( M,E_-)\,.
$$

\medskip
In order to apply the Nash-Moser theorem, we  have to first show that the spaces $\mathcal{F}$ and $\mathcal{G}$ are tame Fr\'echet with respect to the gradings $\|\cdot\|_n$ and $\|\cdot\|_n+|\cdot|_n$,   respectively.  Here is where condition \eqref{G} plays a role. Indeed, the  Green map $G$ of $\Delta_D$ is smooth tame in view of theorem 3.3.3 of \cite{HamiltonNash} and then $\mathcal{F}$ is a direct summand of $C^{\infty}(M\times [0,T],E_-)$ which is a tame Fr\' echet space with respect to $\|\cdot\|_n$. Therefore $\mathcal{F}$ is a tame Fr\' echet space being a direct summand of a tame Fr\' echet space (see \cite[lemma 1.3.3, p. 136]{HamiltonNash}). The same argument can be used to show that $DC^{\infty}( M,E_-)$ is a tame Fr\' echet space with respect to $|\cdot|$ which implies that $(\mathcal{G},\|\cdot\|_n+|\cdot|_n)$ is tame Fr\'echet.

Moreover, $F_{*|\beta}$ is bijective for every $\beta\in\mathcal U\,.$ This can be shown as follows:\\
Equation $F_{*|\beta}(\psi)=(0,0)$ is equivalent to
$$
\begin{cases}
& \frac{\partial}{\partial t} \psi=L_{*|\varphi_0 +\beta}(\psi)\\
& \psi(0)=0\,.
\end{cases}
$$
By hypothesis $L_{*|\varphi_0 +\beta}$ is the restriction of the strongly elliptic operator $\tilde L_{\varphi_0+\beta}$ on $C^{\infty}(M,E)$ and  therefore for every $\beta$ the only solution of the last system is $\psi\equiv 0$ by uniqueness of solutions to {\em linear} parabolic systems.

Now we show that $F_{*|\beta}$ is surjective. Equation $F_{*|\beta}(\psi)=(\chi,\chi_0)$ can be written as
\begin{equation}
\label{F-1}
\begin{cases}
& \frac{\partial}{\partial t} \psi-\tilde L_{\varphi_0+\beta}(\psi)=\chi\\
& \psi(0)=\chi_0\,.
\end{cases}
\end{equation}
We can write $\psi=D\theta$ and $\chi=D	\eta$, $\chi_0=D\eta_0$ so that system \eqref{F-1} writes as
$$
\begin{cases}
& \frac{\partial}{\partial t} D\theta-\tilde L_{\varphi_0+\beta}(D\theta)=D\eta\\
& D\theta(0)=D\eta_0\,.
\end{cases}
$$
Using hypothesis 3, this may be written as
$$
\begin{cases}
& \frac{\partial}{\partial t} D \theta-D l_{\varphi_0+\beta}(\theta)=D\eta\\
&D\theta(0)=D\eta_0\,.
\end{cases}
$$
Since $l_{\varphi_0+\beta}$ is by hypothesis strongly elliptic, the system
$$
\begin{cases}
& \frac{\partial}{\partial t}  \theta- l_{\varphi_0+\beta}(\theta)=\eta\\
&\theta(0)=\eta_0\,.
\end{cases}
$$
has a unique solution for every $\eta_0$.  This implies  that $F_{*|\beta}$ is surjective.

The next step consists in showing that the family of the inverses
$$
VF\colon \mathcal{U}\times \mathcal{G}\to \mathcal F
$$
defined as
$$
VF(\beta,\chi,\chi_0)=F_{*|\beta}^{-1}(\chi,\chi_0)
$$
is a smooth tame map.
Let $(\chi,\chi_0)\in \mathcal{G}$, then $F^{-1}_{|_*\beta}(\chi,\chi_0)$ is the unique solution to \eqref{F-1}. Since equation \eqref{F-1} is parabolic
%by theorem 6.2 of \cite{ham2} (or more precisely by its corollary 4.6 in \cite{BryantXu})
we have
$$
VF(\beta,\chi,\chi_0)=S(\tilde L_{\varphi_0+\beta},\chi,\chi_0)
$$
for a map $S$  which is tame in the gradings $\|\cdot\|_n+|\cdot|$ on $(\chi, \chi_0)$ and
$|[\cdot ]|_n$ on $\tilde L_{\varphi_0+\beta}$. This fact is again a consequence of the ellipticity of $\widetilde{L}$ and will be proved in the next section.\\
Now we can write
$$
VF=S\circ (\tilde L\circ i \times {\rm Id})
$$
where $i$ is the translation
$\mathcal U\hookrightarrow C^{\infty}(M\times [0,T],A)\colon \beta \mapsto \varphi_0+\beta$ and ${\rm Id}\colon \mathcal G\to \mathcal G$ is the identity. Hence $VF$ is a smooth tame map since it is a composition of smooth tame maps. Now we can apply the Nash-Moser theorem to $F$ obtaining that $F$ is a locally invertible map whose local inverse is also a smooth tame map. This implies that if equation
$$
\begin{cases}
\frac{\partial }{\partial t}\beta-L(\varphi_0+\beta)=\bar \chi\\
\beta(0)=\bar \beta_0
\end{cases}
$$
has a solution $\beta\in \mathcal{U}$, then for every $(\chi,\beta_0)$ sufficiently close to
$(\bar \chi, \bar \beta_0)$ the equation
$$
\begin{cases}
\frac{\partial}{\partial t}\beta-L(\varphi_0+\beta)=\chi\\
\beta(0)=\beta_0
\end{cases}
$$
has a unique solution $\beta\in \mathcal{U}$. Now let us solve formally for the derivative $\frac{\partial^k}{\partial t^k}\beta(0)$ by differentiating through the equation.  Let $\bar\beta(t)$ be an element of $\mathcal{U}$ whose Taylor series expansion at $t = 0$ is given by $\frac{\partial^k}{\partial t^k}\beta(0)$. Such a $\bar \beta$ exists in view of a classical theorem of Borel (see e.g. \cite{hormander}, page 16). Then
$$
\bar{\chi }:=\frac{\partial}{\partial t}\bar\beta-L(\varphi_0+\bar \beta)
$$
has trivial Taylor series at $t=0$. Because of this fact $\bar \chi$ can be extended smoothly to $0$ backward in time. Hence
by translating $\bar \chi$ in $t$ we can find a map $\chi$ arbitrarily close to  $\bar \chi$ and vanishing in
a small neighborhood of $0$. In view of the Nash-Moser theorem there exists a solution  $\beta\colon [0,\epsilon)\to DC^{\infty}(M,E_{-1})$ to \eqref{betaflow}. Hence for $t$ small enough $\varphi(t)=\varphi_0+\beta(t)$ is  a solution to \eqref{flow}. Finally, if $\epsilon$ is fixed, then \eqref{betaflow} has at most one solution
$\beta\colon [0,\epsilon)\to DC^{\infty}(M,E_{-1})$ and system \eqref{flow} has a unique solution defined in a maximal interval.
\end{proof}

\subsection{Tame smoothness of families of solutions}\label{tame}
%\textcolor{red}{In this section we show that the solution map to a parabolic system is a smooth tame map. %Although this point is quite standard, we write a complete proof for the reader's convenience. Our proof %follows the Hamilton's proof of theorem 6 of \cite{positive}}
Let $M$ be a compact manifold and let $E$ be a smooth vector bundle over $M$. Let
$P\colon C^{\infty}(M\times [0,T],E)\to C^{\infty}(M\times [0,T],E)$ be a smooth family of strongly elliptic partial differential operators of order $2m$ involving only space derivatives whose coefficients are smooth functions of both space and time. Consider the {\em parabolic} equation
\begin{equation}\label{basic}
\frac{\partial}{\partial t} f=P(f)+h\,.
\end{equation}
where $h$ is a smooth section of $E$ depending on $t$. It is well known that equation \eqref{basic} has a unique solution $f \in C^{\infty}(M\times [0,T],E)$ having initial condition $f_0$ at $t=0$: we denote it by
$$
f=S(P,h,f_0)\,.
$$

\medskip
The goal of this section is to justify the following proposition which plays a major role in the proof of theorem \ref{general1}:

\begin{prop}\label{tame}
Let
$
f=S(P,h,f_0)
$
be the solution of equation \eqref{basic}  having initial condition $f_0$ at $t=0$.
In the open set where $P$ is strongly elliptic the solution $S$ is a smooth tame map in the gradings
$\|\cdot\|_n$  on $f$, $h$ and $|\cdot |_n$ on $f_0$ and $|[\cdot]|_n$  on $P$.
\end{prop}
For $m=1$, proposition \ref{tame} is a consequence of theorem 6.2 of \cite{positive} (see also \cite{BryantXu}). For $m>1$  things work much in the same way and the proof can be obtained combining the standard theory of parabolic equations with some devices described in \cite{HamiltonNash}. Anyway, for reader's convenience, we write down some details in order to show how things work, following the Hamilton proof of theorem 6 of \cite{positive}.

The starting point is the following standard a priori estimate for the solution of the evolution equation \eqref{basic} (see for example \cite{Polden} for the scalar case).\\
%The proof of proposition \ref{tame} follows from the following four lemmas:
%\begin{lemma}
%Let $f=S(P,h,f_0)$ and fix $0\leq \theta\leq T$. Then there exists a constant $C$ non-depending on $\theta$ such that
%$$
%\int_{0}^{\theta} |f_t|^2_{2m}\,dt\leq C \int_{0}^{\theta} |h_t|^2_0\,dt+C |f_0|_m\,.
%$$
%\end{lemma}
%\begin{proof}
%\textcolor{red}{Manca}
%First of all the maximal regularity, the
%\end{proof}
\begin{theorem}\label{apriori}
If $f=S(P,h,f_0)$ there exists a constant $C$ depending on $P$ such that
%and $0\leq \theta\leq T$%
$$
|f|_{2m}\leq C \left( |h|_0 + |f_0|_{m}\right)\,.
$$
%where $C$ does not depend on $\theta$.
\end{theorem}
%\begin{proof}
%\textcolor{red}{Manca}
%$$
%\int_{0}^{\theta} |f_t|^2_{2m}\,dt\leq C^2\left( |h_t|^2_0 + |f_0|_{m}\right)^2
%$$
%\end{proof}

Fix now an elliptic operator $\ov{P}$ and consider the operators $P$ in a neighborhood $[P-\ov{P}]_0\leq \delta$. If $\delta >0$ is small enough then $P$ is again strongly elliptic.
\begin{lemma}\label{7.4}
If $\delta >0$ is small enough, then for all $P$ such that $[P-\ov{P}]_0\leq \delta$ the solution of \eqref{basic} satisfies the estimates
$$
|f|_{2m}\leq C\left(|h|_0+|f_0|_{m}\right)
$$
for a constant $C$ independent of $P$.
\end{lemma}
\begin{proof}
If $f=S(P,h,f_0)$, then $f$ is also a solution of the evolution equation
$$
\begin{cases}
\frac{\partial}{\partial t}g=\ov{P}g+(P-\ov{P})f+h\\
g(0)=f_0
\end{cases}
$$
Applying theorem \ref{apriori} to the fixed operator $\ov{P}$ we have
$$
\begin{aligned}
|f|_{2m}\leq & C\left(|(P-\ov{P})f+h|_0+|f_0|_{m}\right)\leq C\left(|(P-\ov{P})f|_0+|h|_0+|f_0|_{m}\right)\\
        \leq & C \left([P-\ov P]_0|f|_m+|h|_0+|f_0|_{m}\right)
\end{aligned}
$$
which implies the result.
\end{proof}
\begin{lemma}\label{key} For all solutions $f=S(P,h,f_0)$ with $P$ in the $\delta$-neighborhood above and every $k \geq 0$ there exists a constant $C$ such that
$$
|f|_{k+2m}\leq C\left(|h|_k+|f_0|_{k+m}\right)+C\,[P]_k \left( |h|_0+|f_0|_{m}\right)\,.
$$
\end{lemma}
\begin{proof}
The proposition is obtained by induction on $k$. The starting step  $k=0$ is
$$
|f|_{2m}\leq (C+C\,[P]_0) \left( |h|_0+|f_0|_{m}\right)
$$
which is ensured by lemma \ref{7.4}. Assume that the estimate holds up to some  $k$ and let $f$ be a solution to \eqref{basic} with initial condition $f_0$ at $t=0$. We will let $\sum_v$ denote the sum over a finite number of vector fields which span the tangent space at each point of $M$. This is possible since $M$ is supposed to be compact. Now choosing an arbitrary connection $\N$ on $E$ we have
$$
\begin{aligned}
|f|_{k+2m+1}\leq& C\sum_v|\nabla_vf|_{k+2m}+|f|_{k+2m}\\
                     \leq& C\sum_v|\nabla_vf|_{k+2m}+C\left(|h|_k+|f_0|_{k+m}\right)+C\,[P]_k \left( |h|_0+|f_0|_{m}\right)\\
 \leq& C\sum_v|\nabla_vf|_{k+2m+1}+C \left(|h|_{k+1}+|f_0|_{k+1+m}\right)\\
 & + C\,[P]_{k+1} \left( |h|_0+|f_0|_{m}\right)\,.
 \end{aligned}
$$
Therefore in order to prove the statement, we need to estimate the space derivatives of $f$.
For every vector field $v$, the section $\nabla_v f$ gives the solution to the parabolic evolution equation
$$
\begin{cases}
\frac{\partial}{\partial t}  g=P g+(\nabla_v P)f+\nabla_vh\\
g(0)=\nabla_vf_0
\end{cases}
$$
and by the inductive assumption we have
\begin{multline*}
|\nabla_vf|_{k+2m}\leq C\left(|(\nabla_v P)f+\nabla_vh|_k+|\nabla_v f_0|_{k+m}\right)+C\,[P]_k \left( |(\nabla_v P)f+\nabla_vh|_0+|\nabla_vf_0|_{m}\right)\,.
\end{multline*}
%which implies
%\begin{multline}\l{chievo}
%|\nabla_vf|_{k+2m}\leq  C\big([P]_{1}|f|_{k+2m}+[P]_{k+1}|f|_{2m}+|h|_{k+1}+|f_0|_{k+m+1}\\
%+[P]_k[P]_1 |f|_0+[P]_k\left(|h|_1+|f_0|_{m+1}\right)\big)\,.
%\end{multline}
%and
%\begin{multline}\l{chievo}
%|\nabla_vf|_{n+2m}\leq  C\big([P]_{n}|f|_{n+2m}+|h|_{n+1}+|f_0|_{n+m+1}\\+[P]_n[P]_1 |f|_0+[P]_n\left(|h|_1+|f_0|_{m+1}\right)\big)\,.
%\end{multline}
Using the following interpolation estimates
\begin{eqnarray*}
&& |(\nabla_v P)f|_k \leq  C([P]_{1}|f|_{k+2m}+[P]_{k+1}|f|_{2m})\,;\\
&& [P]_1[P]_k\leq C [P]_0[P]_{k+1}\leq C [P]_{k+1}\,;\\
&& [P]_k(|h|_1+|f_0|_{m+1})\leq C[P]_{k+1}(|h|_0+|f_0|_{m})+C[P]_0(|h|_{k+1}+|f_0|_{k+m+1});
\end{eqnarray*}
we reduce to
\begin{multline}\l{spesa}
|\nabla_vf|_{k+2m}\leq  C\big([P]_{1}|f|_{k+2m}+[P]_{k+1}|f|_{2m}+|h|_{k+1}+|f_0|_{k+m+1}\\
+[P]_{k+1} |f|_0+[P]_{k+1}\left(|h|_0+|f_0|_{m}\right)\big)\,.
\end{multline}
%
%and
%\begin{equation}\label{inter?}
%|\nabla_vf|_{n+2m}\leq  C\left([P]_{1}|f|_{n+2m}+|h|_{n+1}+|f_0|_{n+m+1}+[P]_{n+1}|f|_{2m} \right)\,.
%\end{equation}
On the other hand taking the inductive assumption on $f$, multiplying by $[P]_1$ and using again interpolation we have
%$$
%|f|_{k+2m}\leq C(|h|_k+|f_0|_{k+m}+[P]_k \left( |h|_0+|f_0|_{m}\right))
%$$
%we get
\begin{equation}\label{udinese}
[P]_{1}|f|_{k+2m}\leq C [P]_{1}(|h|_k+|f_0|_{k+m})+C[P]_{k+1} \left( |h|_0+|f_0|_{m}\right))\,.
\end{equation}
By interpolation we have
$$
[P]_1(|h|_k+|f_0|_{k+m})\leq C [P]_0(|h|_{k+1}+|f_0|_{k+m+1})+C [P]_{k+1}(|h|_0+|f_0|_{m})\,,
$$
therefore \eqref{udinese} reduces to
$$
\begin{aligned} % qua c' un baco clamoroso del LaTex!!!
{[P]}_1|f|_{k+2m}  \leq & C [P]_0(|h|_{k+1}+|f_0|_{k+m+1})+C[P]_{k+1} \left( |h|_0+|f_0|_{m}\right)\\
\leq & C (|h|_{k+1}+|f_0|_{k+m+1})+C[P]_{k+1} \left( |h|_0+|f_0|_{m}\right)
\end{aligned}
$$
which yields an estimate for the first terms of the right hand side of \eqref{spesa}. The second term $[P]_{k+1}|f|_{2m}$  is readily estimated by means of
$$
[P]_{k+1}|f|_{2m}\leq C[P]_{k+1} (|h|_0+|f_0|_m)\,.
$$
Therefore
$$
|\nabla_vf|_{k+2m}\leq  C\left(|h|_{k+1}+|f_0|_{k+m+1}+[P]_{k+1} (|h|_0+|f_0|_m) \right)
$$
and the claim follows.
\end{proof}
The last step is the estimate for time derivatives.
\begin{lemma}
For all solutions $f=S(P,h,f_0)$ with $P$ in the $\delta$-neighborhood above and every $k \geq 0$ there exists a constant $C$ such that
$$
\|f\|_{k+2m}\leq C\left(\|h\|_k+|f_0|_{k+m}\right)+C\,|[P]|_{k}\left(\|h\|_0+|f_0|_k \right)\,.
$$
\end{lemma}
\begin{proof}
We will estimate the quantity $|(\partial/\partial t)^jf|_{k+2m-2mj}$ by induction on $j$. For $j=0$ we can use lemma \ref{key}. Assume then to have up to some
$j$
\begin{multline*}
 |(\partial/\partial t)^jf|_{k+2m-2mj}\leq C(|(\partial/\partial t)^jh|_{k-2mj}+|(\partial/\partial t)^jf_0|_{k+m-2mj})+C|[P]|_{k}(|h_0|+|f_0|_1).
\end{multline*}
Now
$$
\begin{aligned}
|(\partial/\partial t)^{j+1}f|_{k-2mj}=|(\partial/\partial t)^j(Pf+h)|_{k-2mj}
\end{aligned}
$$
and by interpolation we get
\begin{multline*}
|(\partial/\partial t)^{j+1}f|_{k-2mj}\leq C ([(\partial/\partial t)^jP]_{k-2mj}|f|_{k-2mj+2m}\\+[P]_{k-2mj}|(\partial/\partial t)^jf|_{k-2mj+2m}+|(\partial/\partial t)^{j}h|_{k-2mj} )\,,
\end{multline*}
which implies the statement.
\end{proof}

%%%%%%%%%%%%%% PROOF MAIN %%%%%%%%%%%%%%

\section{Proof of theorem $\ref{main}$}
We apply theorem \ref{general1} to the Hodge system we describe below. We consider
$$
E_-=\Lambda^{n-2,n-2}_\R\xrightarrow{D=i\partial\op} E=\Lambda^{n-1,n-1}_\R\,, %\xrightarrow{D_+=d} E_+=\Lambda^{2n-1}\,,
$$
where $\Lambda^{p,p}_\R$ is the bundle of {\em real} $(p,p)$-forms; the subset $U$ is the set of smooth sections of  $\Lambda^{n-1,n-1}_+$ lying in the same cohomology class as $\varphi_0$
and
$$
L\colon C^{\infty}(M,\Lambda^{n-1,n-1}_+)\to C^{\infty}(M,\Lambda^{n-1,n-1}_\R)
$$
is the operator
$$
L(\varphi)=i\partial\op*(\rho\wedge *\varphi)+(n-1){\Delta}_{BC}(\varphi)\,.
$$

We prove that for every closed $\f$ in $C^{\infty}(M,\Lambda^{n-1,n-1}_+)$ and every closed $\psi$ in $C^{\infty}(M,\Lambda^{n-1,n-1}_\R)$ we have
\begin{equation}\label{L*}
L_{*|\f}(\psi)=(1-n)\,\Delta_{BC}\psi+i\partial\op \Phi_{\f}(\psi)
\end{equation}
where $\Phi_{\varphi}$ is a linear algebraic operator on $\psi$ with coefficients depending on the torsion of $\f$ in a universal way. In particular this shows that $L_{*|\f}$ is the restriction of a strongly elliptic operator to closed $(n-1,n-1)$-forms, since it has the same symbol as $(1-n)\Delta_{BC}\,$ which is strongly elliptic.

\medskip
As recalled in section \ref{the chern connection}, the curvature form $\rho$ of the Chern connection of an Hermitian manifold $(M,\omega,J)$ may be locally written as $\rho=i\rho_{k\bar l} dz^k\wedge d\bar{z}^l$ where
$$
\rho_{k \bar{l}}=-\frac{\partial^2}{\partial z^k \partial \bar{z}^l} \,{\rm log}(G)\,,
$$
$G$ being the determinant of the matrix $g=g_{i\bar j}$.
In the following $\dot{\alpha}$ will denote the derivative with respect to time of the tensor $\alpha$.

\begin{lemma}\label{varrho}
The derivative of $\rho$ is
$$
\dot{\rho}=-i\partial\op\,(\omega,\dot{\omega})
$$
where $(\cdot,\cdot)$ denotes the pointwise scalar product of $(1,1)$-forms.
\end{lemma}
\begin{proof}
Using
$$
\frac{\partial}{\partial t}{\rm det}\,g=({\rm det}\, g)\,{\rm tr}\left(g^{-1}\dot{g}\right)
$$
we get
$$
\dot \rho_{k\bar l}=-\frac{\partial^2}{\partial z^k \partial \bar{z}^l}\,\,{\rm tr}\,\left(g^{-1}\dot{g}\right)\,.
$$
Now it is enough to observe that
$$
{\rm tr}\left(g^{-1}\dot{g}\right)=g^{i\bar j}\dot{g}_{i\bar j}=(\omega,\dot{\omega})\,.
$$
\end{proof}
\noindent We write the operator $L$ as
$$
L=P+Q
$$
where
$$
P(\varphi)=i\partial\op*(\rho\wedge *\varphi)\,,\quad Q(\varphi)=(n-1){\Delta}_{BC}(\varphi)
$$
and compute the first derivative of $P$ and $Q$ separately.
Let $\psi=\frac{d}{dt}\varphi$ be a tangent vector to $U$ at $\varphi$. Then we can write
$$
\psi=h_1\varphi+*h_0
$$
where $h_1$ is a smooth function on $M$ and $h_0$ is a section of $\Lambda_{0}^{1,1}$.
Note that
$$
h_1=\frac{1}{n}(\omega,\dot{\omega})=\frac{1}{n}(\varphi,\dot{\varphi})\,.
$$
{
The derivative of $P$ is now obtained using lemma \ref{varrho}:
$$
P_{*|\varphi}(\psi)= P_{*|\varphi}(\psi)=i\d\op\,*\left(\dot{\rho}\wedge *\varphi\right)+i\d\op\Phi_1(\psi) =n\d\op\,*\left(\d\op h_1\wedge *\varphi\right)+i\d\op\Phi_1(\psi)\,,
$$
where $\Phi_1$ is an algebraic operator depending on $\varphi$ in a universal way. \\
As for the derivative of $Q$,  using lemma \ref{firstvariation} we have
$$
Q_{*|\varphi}(\psi)=Q_{*|\varphi}(h_1\varphi+*h_0)=-\partial \op*\left(\partial \op h_1\wedge* \varphi+(1-n)\partial \op h_0\right)+ i\d\op\Phi_2(\psi)\,,
$$
for a suitable linear zeroth order operator $\Phi_2$.

Summing up we have
$$
\begin{aligned}
L_{*|\varphi}(\psi)=(n-1)\partial \op*\left(\partial \op h_1\wedge* \varphi+\partial \op h_0\right)+i\d\op\Phi_1(\psi)+i\d\op\Phi_2(\psi)\,.
\end{aligned}
$$
On the other hand using that $\psi$ is closed we have
$$
\begin{aligned}
(\Delta_{BC})_\varphi\psi=&\,\partial\ov{\partial}\ov{\partial}^*\partial^*\psi=\partial\ov{\partial}\ov{\partial}^*\partial^*(h_1\varphi+*h_0)\\
=&\,-\partial \op*\left(\partial \op h_1\wedge* \varphi+\partial \op h_0\right)+ i\d\op\Phi_3(\psi)
\end{aligned}
$$
and therefore
$$
L_{*|\varphi}(\psi)=-(n-1)(\Delta_{BC})_\varphi\psi+i\d\op\Phi_{\varphi}(\psi)
$$
where $\Phi_\varphi=\Phi_1+\Phi_2-\frac{1}{n-1}\Phi_3$.
Hence equation \eqref{L*} is established.}
\medskip
In order to verify the last hypothesis of theorem \ref{general1} let us consider
$$
l_\varphi=-(n-1)(\Delta_A)_\varphi+i \Phi_\varphi \circ \partial\op\,,
$$
where $\Delta_{A}$
is the modified Aeppli Laplacian defined in section \ref{BCsection}.
Here it is enough to recall that $-\Delta_A$ is strongly elliptic and
$\Delta_{BC}(\partial\ov{\partial}\theta)=\d\op\left(\Delta_A\theta\right)$ for every form $\theta$.
%\textcolor{red}{This show that system \eqref{main} has a unique solution $\beta(t)$.(\textcolor{cyan}{anche questa frase aggiunta non serve. E poi chi e' $\beta$?)}}

\medskip
 The last step of the proof consists in showing that if $\varphi_0$ is the $(n-1,n-1)$-positive form of a K\"alher structure, then the solution $\varphi_0+\beta(t)$ to \eqref{theflow} corresponds to a family of K\"ahler forms $\omega(t)$ solving the Calabi flow. \\
Let $(M,J,\omega_0)$ be a compact K\"ahler manifold and let $\omega(t)$ be a solution to the Calabi flow
\begin{equation}\label{calabi}
\begin{cases}
\frac{\partial}{\partial t}\omega(t)=i\partial\op\,s_t\\
\omega(0)=\omega_0\,
\end{cases}
\end{equation}
where $s_t$ is the scalar curvature of metric $\omega(t)$. Since in the K\"ahler case the Levi-Civita and the Chern connection coincide, we can consider $s$ as the scalar curvature of the Chern connection. Let $\varphi(t)=*_t\omega(t)=\frac{1}{(n-1)!}\omega(t)^n$. Then
$$
\frac{d}{dt}\varphi=\frac{1}{(n-1)!}\,\frac{d}{dt}(\omega^{n-1})=\frac{1}{(n-2)!}\,\left(\frac{d}{dt}\omega\right)\wedge \omega^{n-2}
$$
and using \eqref{calabi} we get
$$
\frac{d}{dt}\varphi=\frac{1}{(n-2)!} \left(i\partial\op\,s \right)\wedge \omega^{n-2}\,.
$$

\begin{lemma}\label{lemmacalabi}
Let $(M,J,\omega)$ be a K\"ahler  manifold. Then the following identity holds
$$
\frac{1}{(n-2)!} \left(i\partial\op\,s \right)\wedge \omega^{n-2}=i\partial\op*(\rho\wedge *\varphi)\,,
$$
where $\varphi=*\omega$ and $s$ and $\rho$ are the scalar curvature and the Ricci form, respectively.
\end{lemma}
\begin{proof}
The proof of this lemma is a straightforward computation taking into account that $\rho$ and $\omega$ are closed. Indeed
$$
\partial\op*(\rho\wedge *\varphi)=\partial\op*\left(\frac{s}{n}\omega^2+\rho_0\wedge \omega\right)\,,
$$
where we have written $\rho = \frac{s}{n}\omega + \rho_0$ according to the decomposition \eqref{dec(1,1)}.
Using lemma \ref{star(1,1)}, the closure of $\omega$ and the formula
$$
*\,\omega^2=\frac{2}{(n-2)!}\,\omega^{n-2}\,
$$
we have
$$
\partial\op*(\rho\wedge *\varphi)=\frac{2}{n(n-2)!}\partial\op s\wedge \omega^{n-2}-\frac{1}{(n-3)!}\partial\op \rho_0\wedge\omega\,.
$$
Moreover, taking into account that $\rho$ is closed, we get
$$
\partial\op(\rho_0)=-\partial\op\left(\frac{s}{n}\,\omega\right)
$$
and therefore
$$
\begin{aligned}
\partial\op*(\rho\wedge *\varphi)
=\frac{1}{(n-2)!}\partial\op s\wedge \omega^{n-2}\,,
\end{aligned}
$$
as required.
\end{proof}
Now we can conclude the proof of theorem \ref{main} showing that if $\omega(t)$ is a solution of \eqref{calabi}, then the corresponding $\varphi(t)$ solves \eqref{theflow}. In view of lemma \ref{lemmacalabi}, $\varphi(t)$ satisfies
$$
\frac{d}{dt}\varphi=i\partial\op*(\rho\wedge *\varphi)
$$
and since $\omega(t)$ is closed, we have $\Delta_{BC}\varphi=0$ and $\varphi(t)$ in particular satisfies
$$
\frac{d}{dt}\varphi=i\partial\op*(\rho\wedge *\varphi)+(n-1)\Delta_{BC}\varphi\,.
$$
Hence $\varphi$ solves \eqref{theflow} and the statement follows from the uniqueness of solutions.

 %%%%%%%%%%%%%%%%%%%%%%%%%%%%%%%%%%%%%%%

\section{The flow on the Iwasawa manifold}
In this section we study flow \eqref{theflow} in the case of the Iwasawa manifold. 
The Iwasawa manifold is the compact complex manifold defined as the quotient
$$
M={\rm H}_{3}(\C)\slash \Gamma
$$
where ${\rm H}_{3}(\C)$ is the $3$-dimensional complex Heisenberg group
$$
{\rm H}:=\left\{
\left(
\begin{array}{cccccc}
1  &z^1    &z^2  \\
0  &1      &z^3     \\
0  &0      &1\\
\end{array}
\right)\,:z^k\in\C\,,k=1,2,3
\right\}
$$
and $\Gamma$ is the co-compact lattice
$$
\Gamma:=\left\{
\left(
\begin{array}{cccccc}
1  &z^1    &z^2  \\
0  &1      &z^3     \\
0  &0      &1\\
\end{array}
\right)\in {\rm H}\,:z^k\in\Z\oplus i \Z\,,k=1,2,3
\right\}
$$
(see e.g. \cite{elsa}). We denote by $J$ the natural complex structure on $M$ induced by ${\rm H}_3(\C)$ and we set
$$
\alpha^1=dz^1\,,\quad \alpha^2=dz^2\,,\quad \alpha^3=-dz^3+z^1dz^2\,.
$$
Then $\{\alpha^1,\alpha^2,\alpha^3\}$ is a global $(1,0)$-coframe satisfying the structure equations
$$
d\alpha^1=d\alpha^2=0\,,\quad d\alpha^3=\alpha^1\wedge \alpha^2\,.
$$
It is immediate to verify that every left-invariant Hermitian metric on $(M,J)$ is balanced.

\medskip
Let us study flow \eqref{theflow} with initial condition
$$
\omega_0=i\alpha^{1\bar 1}+i\alpha^{2\bar 2}+i\alpha^{3\bar 3}\,.
$$
(Here we adopt the notation $\alpha^{ijk\dots}=\alpha^i\wedge \alpha^j\wedge \alpha^{k}\wedge \dots$).
Since \eqref{theflow} is invariant by biholomorphisms and $\omega_0$ is left-invariant, the solution $\omega(t)$ to the flow has to be left-invariant. On the other hand, from \cite{luigiproc} it follows  that a left-invariant Hermitian metric on a $2$-step nilmanifold has always vanishing Ricci form. Therefore in this case the flow \eqref{theflow} reduces to
$$
\begin{cases}
\frac{d}{dt}\varphi=2\Delta_{BC}\varphi\\
\varphi(0)=\frac{1}{2}\omega_0^2
\end{cases}
$$
which can be alternatively rewritten in term of $2$-forms as
\begin{equation}\label{flowiwa}
\begin{cases}
\frac{d}{dt}\omega=\,\iota_{\omega}\Delta_{BC}\omega^2\\
\omega(0)=\omega_0
\end{cases}
\end{equation}
$\iota_{\omega}$ denoting the contraction along $\omega$. We seek for a solution to the flow taking the following diagonal expression
$$
\omega=ig_1\,\alpha^{1\bar 1}+ig_2\,\alpha^{2\bar 2}+ig_3\,\alpha^{3\bar 3}
$$
where $g_1,g_2,g_3$ are real numbers depending smoothly on $t$. Now,
$$
\Delta_{BC}\omega^2=-\d\op*\d\op*(\omega^2)=-2\d\op*\d\op\,\omega
$$
and since
$$
\d\op\,\omega=-ig_{3}\alpha^{1\bar 1 2\bar 2}\,,\quad *\alpha^{1\bar 1 2\bar 2}=-i\frac{g_3}{g_1g_2}\alpha^{3\bar 3}
$$
we have
$$
\Delta_{BC}\omega^2=-\frac{g_3^2}{g_1g_2}\d\op(\alpha^{3\bar 3})=\frac{g_3^2}{g_1g_2}\,\alpha^{1\bar 1 2\bar 2}\,.
$$
Since
$$
\iota_{\omega}\left(\frac{g_3^2}{g_1g_2}\,\alpha^{1\bar 1 2\bar 2}\right)=-i\frac{g_3^2}{g_1g_2^2}\alpha^{1\bar 1}-i\frac{g_3^2}{g_1^2g_2}\alpha^{2\bar 2}+i\frac{g_3^3}{g_1^2g_2^2}\alpha^{3\bar 3}
$$
then system \eqref{flowiwa} reads as
$$
\begin{cases}
\dot{g}_1=-\frac{g_3^2}{g_1g_2^2}\\
\dot{g}_2=-\frac{g_3^2}{g_1^2g_2}\\
\dot{g}_3=\frac{g_3^3}{g_1^2g_2^2}
\end{cases}
$$
with initial condition
$$
g_1(0)=g_2(0)=g_3(0)=1\,.
$$
This last system has the solution
$$
g_1(t)=g_2(t)=\sqrt[6]{1-6t}\,,\quad g_3(t)=\frac{1}{\sqrt[6]{1-6t}}
$$
and
$$
\omega(t)=\sqrt[6]{1-6t}\,\alpha^{1\bar 1}+\sqrt[6]{1-6t}\,\alpha^{2\bar 2}+\frac{1}{\sqrt[6]{1-6t}}\,\alpha^{3\bar 3}
$$
solves \eqref{flowiwa}.

\medskip
So this is an example where the flow admits an ancient non-eternal solution diverging for $t\to -\infty$.

%%%%%%%%%%%%%%%%%%%%%%%%%%%%%%%%%% BIBLIOGRAFIA%%%%%%%%%%%%%%%%%%%%%%%%%%%%%%%%


\begin{thebibliography}{12}

\bibitem{elsa}
E. Abbena, S. Garbiero, S. M. Salamon:
Almost Hermitian geometry on six dimensional nilmanifolds.
{\em Ann. Scuola Norm. Sup. Pisa Cl. Sci.} (4) {\bf 30} (2001), no. 1, 147--170.

\bibitem{AB1}
L. Alessandrini, G. Bassanelli: Small deformations of a class of compact non-K\"ahler manifolds. {\em Proc. Amer. Math. Soc.} {\bf 109} (1990), no. 4, 1059--1062.

\bibitem{AB2}
L. Alessandrini, G. Bassanelli: Metric properties of manifolds bimeromorphic to compact K\"ahler spaces, {\em J. Differential
Geom.} {\bf 37} (1993), 95--121.

\bibitem{ivanov}
B. Alexandrov, S. Ivanov: Vanishing theorems on Hermitian manifolds. {\em Differential Geom. Appl.} {\bf 14} (2001), no. 3, 251--265.

\bibitem{Angella} D. Angella, A. Tomassini: The $\partial\overline{\partial}$-lemma and Bott-Chern cohomology, {\em Invent. Math.} {\bf 192} (2013), no. 1, 71--81.

\bibitem{bigolin}
B. Bigolin: Gruppi di Aeppli, {\em Ann. Scuola Norm. Sup. Pisa (3)}  {\bf 23} (1969), no. 2, 259--287.

\bibitem{Bryant}
R. Bryant: {\it Some remarks on G$_2$-structures},
Proceedings of Gokova Geometry/Topology Conference, Gokova 2006,  75--109.

\bibitem{BryantXu}
R. Bryant, F. Xu: Laplacian Flow for Closed $G_2$-Structures: Short Time Behavior. {\tt  arXiv:1101.2004}\,.

\bibitem{chen}
X.X. Chen, W. Y. He: On the Calabi flow. {\em Amer. J. Math.} {\bf 130} (2008), no. 2, 539--570.

\bibitem{Demailly}
J.-P. Demailly: {\em Complex Analytic and Differential Geometry},
{\tt on-line book}, 2009.

\bibitem{finosalamonparton}
A. Fino, M. Parton, S. M. Salamon: Families of strong KT structures in six dimensions. {\em Comment. Math. Helv.} {\bf 79} (2004), no. 2, 317--340.

\bibitem{fuyau}
J. Fu, S.-T. Yau: A note on small deformations of balanced manifolds. {\em C. R. Math. Acad. Sci. Paris} {\bf 349} (2011), no. 13-14, 793--796.

\bibitem{fuliyau}
 J. Fu, J. Li, S.-T. Yau: Balanced metrics on non-K\"ahler Calabi-Yau threefolds. {\em J. Differential Geom.} {\bf 90} (2012), 81--129.

\bibitem{Gaud}
P. Gauduchon: Fibr\'es hermitiens \`a endomorphisme de Ricci non n\'egatif., {\em Bull. Soc. Math. Fr.} {\bf 105} (1977), 113--140.

\bibitem{Gaudbumi}
P. Gauduchon: Hermitian connections and Dirac operators. {\em Boll. Un. Mat. Ital. B (7)} {\bf 11} (1997), no. 2, suppl., 257--288.

\bibitem{Gill} M. Gill: Convergence of the parabolic complex Monge-Amp\`ere equation on compact Hermitian manifolds. {\em Comm. Anal. Geom.} {\bf 19} (2011), no. 2, 277--303.

\bibitem{GGP}
D. Grantcharov, G. Grantcharov, Y.S. Poon: Calabi-Yau connections with torsion on toric bundles. {\em J. Differential Geom.} {\bf 78} (2008), no. 1, 13--32.

%\bibitem{fernandez}
%M. Fern\'andez, A. Tomassini, L. Ugarte, R. Villacampa: Balanced %Hermitian metrics from SU(2)-structures. {\em J. Math. Phys.} {\bf 50} %(2009), no. 3, 15 pp.

\bibitem{Gri}
S. Grigorian: Short-time behaviour of a modified Laplacian coflow of $G_2$-structures. {\em Adv. Math.} {\bf 248} (2013), 378--415.

\bibitem{HamiltonNash}
R. S. Hamilton:  The inverse function theorem of Nash and Moser,  {\em Bull. Amer. Math. Soc.} (N.S.) {\bf 7} (1982), no. 1, 65--222.

\bibitem{positive}
R. S. Hamilton: Three-manifolds with positive Ricci curvature. {\em J. Differential Geom.} {\bf 17} (1982), no. 2, 255--306.

\bibitem{hormander}
L. H\"ormander: {\em The Analysis of Linear Partial Differential Operators I. Distribution theory and Fourier analysis.} Grundlehren der Mathematischen Wissenschaften, {\bf 256}. Springer-Verlag, Berlin, 1990.

\bibitem{Polden}
G. Huisken, A. Polden: {\em Geometric evolution equations for hypersurfaces.} Calculus of variations and geometric evolution problems (Cetraro, 1996), SpringerÐVerlag, Berlin, 1999, 45--84.


\bibitem{Kar}  S. Karigiannis, B. McKay, M.-P. Tsui: Soliton solutions for the Laplacian coflow of some
${\rm G}_2$ structures with symmetry. {\em Differential Geom. Appl.} {\bf 30} (2012), 318--333.

\bibitem{liwang}
H.-V. Le, G. Wang:
Anti-complexified Ricci flow on compact symplectic manifolds. {\em J. Reine Angew. Math}. {\bf 530} (2001), 17--31.

%\bibitem{lin}
%C. Lin: Laplacian Solitons and Symmetry in ${\rm G}_2$-geometry. {\tt arXiv:1201.2627}.

\bibitem{M}
M. L. Michelsohn:
On the existence of special metrics in complex geometry.
{\em Acta Math.} {\bf 149} (1982), no. 3-4, 261--295.

\bibitem{saracco} A. Saracco, A. Tomassini: On deformations of compact balanced manifolds. {\em Proc. Amer. Math. Soc.} {\bf 139} (2011), no. 2, 641--653.

\bibitem{Schweitzer} M. Schweitzer:  Autour de la cohomologie de Bott-Chern, {\tt arXiv:0709.3528v1}.

\bibitem{streets-tian1}
J. Streets, G. Tian: Hermitian curvature flow. {\em J. Eur. Math. Soc. (JEMS)} {\bf 13} (2011), no. 3, 601--634.

\bibitem{streets-tian2}
J. Streets, G. Tian: A parabolic flow of pluriclosed metrics. {\em Int. Math. Res. Notices} (2010), 3101--3133.

\bibitem{streets-tian3}
J. Streets, G. Tian: Symplectic curvature flow. {\tt  arXiv:1012.2104}. To appear in {\em J. Reine Angew. Math}.

\bibitem{tosatti}
V. Tosatti, B. Weinkove: Estimates for the complex Monge-Amp\'ere equation on Hermitian and balanced manifolds. {\em Asian J. Math.} {\bf 14 } (2010), no. 1, 19--40.

\bibitem{tosatti2} V. Tosatti, B. Weinkove: On the evolution of a Hermitian metric by its Chern-Ricci form, {\tt arXiv:1201.0312}.

\bibitem{tosatti3} V. Tosatti, B. Weinkove:
The Chern-Ricci flow on complex surfaces, {\em Compos. Math.} {\bf 149} (2013), no. 12, 2101--2138.

\bibitem{ugarte}
L. Ugarte: Hermitian structures on six dimensional nilmanifolds, {\em Transfor. Groups.} {\bf 12} (2007),
175--202.

\bibitem{luigiproc} L. Vezzoni: A note on canonical Ricci forms on 2-step nilmanifolds. {\em Proc. Amer. Math. Soc.}, {\bf 41} (2013) 325--333.

\bibitem{luigiflow} L. Vezzoni: On Hermitian curvature flow on almost complex manifolds. {\em Differential Geom. Appl.} {\bf 29} (2011), 709--722.

\bibitem{witt1}  H. Weiss, F. Witt: A heat flow for special metrics. A heat flow for special metrics. {\em Adv. Math.} {\bf 231} (2012), no. 6, 3288--3322..

\bibitem{witt2} H. Weiss, F. Witt: Energy functionals and soliton equations for ${\rm G}_2$-forms. {\em Ann. Global Anal. Geom.} {\bf 42} (2012), no. 4, 585--610.

\bibitem{XuYe}
F. Xu,  R. Ye: Existence, convergence and limit map of the Laplacian Flow, {\tt arXiv:0912.0074}.
\end{thebibliography}
\end{document}